\documentclass[11pt]{amsart}
\usepackage[margin=1.15in]{geometry}
\usepackage{amsmath,amssymb,amsthm}
\usepackage{mathrsfs}
\usepackage{enumitem}
\usepackage{xcolor}
\definecolor{linkblue}{rgb}{0.05,0.20,0.60}
\usepackage[colorlinks=true,linkcolor=linkblue,citecolor=linkblue,urlcolor=linkblue]{hyperref}
\hypersetup{pdftitle={Bernstein-Sato ideals might not be linear},pdfauthor={Lei Wu}}

\theoremstyle{plain}
\newtheorem{theorem}{Theorem}[section]
\newtheorem{proposition}[theorem]{Proposition}

\theoremstyle{definition}

\theoremstyle{remark}
\newtheorem{remark}[theorem]{Remark}

\theoremstyle{plain}
\newtheorem{thmx}{Theorem}
\renewcommand{\thethmx}{\Alph{thmx}}

\newcommand{\sD}{\mathscr{D}}
\newcommand{\cA}{\mathcal{A}}
\newcommand{\cB}{\mathcal{B}}
\newcommand{\cH}{\mathcal{H}}
\newcommand{\cK}{\mathcal{K}}
\newcommand{\cL}{\mathcal{L}}
\newcommand{\cM}{\mathcal{M}}
\newcommand{\cN}{\mathcal{N}}
\newcommand{\cO}{\mathcal{O}}
\newcommand{\cQ}{\mathcal{Q}}
\newcommand{\cP}{\mathcal{P}}
\newcommand{\cS}{\mathcal{S}}
\newcommand{\cU}{\mathcal{U}}
\newcommand{\cV}{\mathcal{V}}
\newcommand{\cC}{\mathcal{C}}
\newcommand{\bone}{\mathbf{1}}
\newcommand{\C}{\mathbb{C}}
\newcommand{\Q}{\mathbb{Q}}
\newcommand{\Z}{\mathbb{Z}}
\newcommand{\N}{\mathbb{N}}
\newcommand{\bA}{\mathbb{A}}
\newcommand{\bP}{\mathbb{P}}
\newcommand{\fm}{\mathfrak{m}}
\newcommand{\fp}{\mathfrak{p}}
\newcommand{\Exp}{\operatorname{Exp}}
\newcommand{\Ann}{\operatorname{Ann}}
\newcommand{\Frac}{\operatorname{Frac}}
\newcommand{\Hom}{\operatorname{Hom}}
\newcommand{\DR}{\operatorname{DR}}
\newcommand{\Res}{\operatorname{Res}}
\newcommand{\supp}{\operatorname{supp}}
\newcommand{\Ch}{\operatorname{Ch}}
\newcommand{\CC}{\operatorname{CC}}
\newcommand{\gr}{\operatorname{gr}}
\newcommand{\Spec}{\operatorname{Spec}}
\newcommand{\an}{\mathrm{an}}
\newcommand{\rel}{\mathrm{rel}}
\newcommand{\wt}{\widetilde}
\newcommand{\Drel}{\sD_{X^{\an}\times T/T}}
\newcommand{\dlog}{d\!\log}
\newcommand{\into}{\hookrightarrow}
\newcommand{\onto}{\twoheadrightarrow}

\title[Bernstein--Sato ideals might not be linear]{Bernstein--Sato ideals might not be linear}
\author{Lei Wu}
\address{Lei Wu, School of Mathematical Sciences, Zhejiang University, Hangzhou, 310058, P.R. China}
\email{leiwu23@zju.edu.cn}

\begin{document}

\begin{abstract}
Let $F=(f_1,\dots,f_r)$ be a family of nonzero polynomials on $\C^n$. Budur \cite[Conjecture 1.1]{Bud15} conjectured that $B_F$, the Bernstein--Sato ideal of $F$, is generated by products of linear polynomials with rational coefficients. With the aid of AI, we construct two counterexamples to Budur's conjecture. The examples indicate that main results in \cite{Wu26} provide the \emph{optimal} linearity properties that the zero locus of Bernstein-Sato ideals satisfies in general. 
\end{abstract}

\maketitle
\section{Introduction}

Let $F=(f_1,\dots,f_r)\colon X=\C^n\to\bA^r_\C$ be a non-constant regular map. Write $s=(s_1,\dots,s_r)$ for a tuple of independent variables, and put $f=\prod_{i=1}^r f_i$ and $F^{s}=\prod_{i=1}^r f_i^{s_i}$.
The \emph{Bernstein--Sato ideal} of $F$ is the ideal
\[
B_F=B_F^{\bone}=\bigl\{\,b(s)\in\C[s_1,\dots,s_r]\;\big|\; b(s)F^{s}=P(s)F^{s+\bone}\ \text{for some }P(s)\in\sD_X[s_1,\dots,s_r]\,\bigr\},
\]
where $\bone=(1,\dots,1)\in \N^r$ and $\sD_X$ is the ring of (algebraic) differential operators on $X$. More generally, for a monoid ideal $K\subseteq \N^r$ one defines the Bernstein--Sato ideal $B^K_F$ of $F$ along $K$ (see \cite[\S 4.1]{Bud15} and \cite[\S 6]{Wu26}).
Sabbah \cite{Sab87a,Sab87b} proved that $B_F$ contains a nonzero product of linear polynomials of the form $a_1s_1+\dots+a_rs_r+\alpha$ with $a\in\N^r$ and $\alpha\in\Q$; see also \cite{Gyo93,Bah05}. For $r=1$, the ideal $B_F$ is generated by the classical Bernstein--Sato polynomial, whose roots are negative rational numbers by Kashiwara \cite{Kas76}. We denote by $Z(B_F)\subseteq\C^r$ the zero locus of $B_F$, always considered with its reduced structure, and by
\[
\Exp\colon\C^r\to(\C^*)^r,\qquad (\alpha_1,\dots,\alpha_r)\mapsto(e^{2\pi i\alpha_1},\dots,e^{2\pi i\alpha_r}),
\]
the universal covering map. Budur \cite{Bud15} related $\Exp(Z(B_F))$ to cohomology support loci of rank-one local systems and raised the question of how to interpret $Z(B_F)$ topologically. This question was addressed in \cite{BvdVWZ21a,BvdVWZ21b} and answered completely in \cite[Corollary B]{Wu26} by means of relative and logarithmic $\sD$-modules; in particular, $\Exp(Z(B_F))\subseteq(\C^*)^r$ is a finite union of torsion-translated subtori. Note that the preimage of $\Exp(Z(B_F))$ is a union of linear subvarieties of $\C^r$ defined over $\Q$.

For $B_F$ itself, based on the results of Sabbah and Kashiwara mentioned above and on computational evidence, Budur \cite[Conjecture 1.1]{Bud15} conjectured that $B_F$ is generated by products of linear polynomials of the form $c_1s_1+\dots+c_rs_r+c_0$ with $c_1,\dots,c_r\in\N$ and $c_0\in\Q_{>0}$ (see \cite[Remark 6.8(1)]{BSZ25} for the analogous statement for $B_F^K$). If this were the case, then every irreducible component of $Z(B_F)$ would be cut out by affine linear equations with rational coefficients.

With the help of AI, we found the following two counterexamples.

\begin{thmx}\label{thm:main}
Let $X=\C^2$ with coordinates $x,y$ and consider the ordered tuple
\[
F_8=\bigl(x,\ y,\ x+y,\ x+2y,\ x+3y,\ x+4y,\ x+\sqrt2\,y+y^2,\ x+5y+y^2\bigr).
\]
Put
\[
S=\sum_{i=1}^8 s_i,\qquad L=60(s_7+1)+(49\sqrt2-70)(s_8+1).
\]
Then the six-dimensional affine subspace
\[
C_8=\{s\in\C^8\mid S+12=0,\ L=0\}
\]
is an irreducible component of $Z(B_{F_8})$ whose direction is not defined over $\Q$.
\end{thmx}

Consequently, $C_8\subseteq \C^8$ cannot be cut out by affine linear equations with coefficients in $\Q$, which disproves Budur's conjecture above (as well as its geometric analogue \cite[Conjecture E]{Wu26}). 
See Remark~\ref{rmk:mainrmk} for how the coefficients in Theorem~\ref{thm:main} are chosen to ensure the irrationality of $C_8$. 

Another example, also found with the assistance of AI, was constructed in \cite[Theorem 1.5]{GWX26}: for a line arrangement of degree $4$ defined by linear forms with integral coefficients and a monoid ideal $K\subseteq\N^4$ generated by two vectors, the Bernstein--Sato ideal $B^K_F$ has an associated minimal prime defining an irreducible nonlinear quadric surface in $\C^4$. Hence $B_{F_8}$ and the example of \cite{GWX26} show that, in general, $B_F^K$ need not be $\Q$-linear. On the other hand,
\cite[Theorem 1.5.1]{BvdVWZ21a} and \cite[Theorem 1.1.1]{BvdVWZ21b} for principal monoid ideals $K$, and more generally \cite[Theorems A, C and D]{Wu26} for arbitrary $K$, thus provide the \emph{optimal} linearity properties that the zero locus of $B_F^K$ satisfies in general.

The second example is in $X=\C^3$ with coordinates $x,y,z$. Choose ten complex numbers
\[
\xi_1,\ldots,\xi_5,c_1,\ldots,c_5
\]
algebraically independent over $\overline{\mathbb Q}$. Set
\[
A_i(x)=\prod_{j\ne i}\frac{x-\xi_j}{\xi_i-\xi_j},\qquad
G_i=1+A_i(x)(z+c_i),\qquad
H_i=(x-\xi_i)^5+y^5+G_i(x-\xi_i)^2y^3.
\]

We use the following exponent vectors:
\begin{center}
\begin{tabular}{c|c|c}
\hline
$i$ & $a_i$ & $b_i$ \\
\hline
1 & $(1,1,3)$ & $(0,0,3)$ \\
2 & $(1,1,3)$ & $(0,1,2)$ \\
3 & $(1,2,2)$ & $(0,0,3)$ \\
4 & $(1,2,2)$ & $(0,1,2)$ \\
5 & $(1,3,1)$ & $(0,3,0)$ \\
\hline
\end{tabular}
\end{center}
For $j=1,2,3$, define
\begin{equation}\label{eq:tuple}
f_j=\prod_{i=1}^5G_i^{a_{ij}}H_i^{b_{ij}},
\end{equation}
where $a_{ij}$ and $b_{ij}$ denote the $j$-th coordinates of $a_i$ and $b_i$, and take $F=(f_1,f_2,f_3)$.
\begin{thmx}\label{thm:main2}
The Bernstein--Sato ideal $B_F$ has no
generating set consisting of products of linear polynomials.
\end{thmx}
Theorem~\ref{thm:main2} gives a second counterexample to Budur's conjecture. The construction relies on a computational example of Bahloul and Oaku \cite{BaOK}.

\subsection*{Acknowledgement} Both examples were constructed by GPT-6-Astra. The resulting statements and proofs have been checked, simplified and/or rephrased by the author, who is solely
responsible for the correctness of the results and for the final form of the text.

\section{A Nakayama-type lemma}\label{sec:support}
Let $X=\C^n$ with coordinates $x_1,\dots,x_n$, and let $F=(f_1,\dots,f_r)$ be nonzero polynomials on $X$. We consider the relative meromorphic connection $\cS_F=\C[x_1,\dots,x_n,1/f,s]\,F^s$ and its $\sD_X[s]$-submodules
\[\cN=\sD_X[s]F^s,\qquad \cM=\sD_X[s]fF^s,\]
together with the quotient module $\cQ_F={\cN}/{\cM}$. All of these modules are relative holonomic over $R=\C[s]$ (cf.~\cite{BvdVWZ21a}), and by definition $B_F=\Ann_R(\cQ_F)$.
For $\beta=(\beta_1,\dots,\beta_r)\in \C^r\simeq \Spec R$, we similarly have the meromorphic connection and its submodule
\[\cS_\beta=\C[x_1,\dots,x_n,1/f]\,F^\beta,\qquad \cM_\beta=\sD_X\cdot(fF^\beta)\subseteq \cS_\beta.\]
Since $\cS_F$ is flat over $R$, we have $\cS_\beta\simeq \cS_F\otimes_R \C_\beta$,
where $\C_\beta$ is the residue field of $\beta\in \Spec R$. For a point $x\in X$, we write $B_{F,x}$ for the local Bernstein--Sato ideal of $F$ at $x$.

\begin{proposition}\label{prop:relgen}
With notation as above, if $\cM_\beta=\cS_\beta$ for some $\beta\in \C^r$, then $\beta\notin Z(B_F)$. Similarly, if the germs at a point $x\in X$ satisfy $\cM_{\beta,x}=\cS_{\beta,x}$, then $\beta\notin Z(B_{F,x})$.
\end{proposition}

\begin{proof}
By \cite[Theorem 3.12]{Wu20}, we have 
\[\cS_\beta\simeq \sD_X[s]f^{-k}F^s\otimes_R \C_\beta\]
for some integer $k\gg 0$. Consider the short exact sequence
\[0\to \cM\to \sD_X[s]f^{-k}F^s\to \cP\to0,\]
where $\cP$ denotes the quotient module. Tensoring with $\C_\beta$ gives an exact sequence
\[\cM\otimes_R\C_\beta\to \cS_\beta\to \cP\otimes_R\C_\beta\to 0.\]
Since the image of $\cM\otimes_R\C_\beta$ in $\cS_\beta$ is $\cM_\beta=\cS_\beta$, we get $\cP\otimes_R\C_\beta=0$. By \cite[Theorem E]{vdV21}, this implies
\[
\beta\notin Z\big(\Ann_R(\cP)\big).
\]
By \cite[Lemma 3.2.2.(2) and Lemma 3.4.1]{BvdVWZ21a}, the short exact sequence
\[0\to \cQ_F\to \cP\to \sD_X[s]f^{-k}F^s/\cN\to0\]
thus gives
\[\beta\notin Z(B_F)=Z\big(\Ann_R(\cQ_F)\big).\]
The statement for germs is proved in the same way.
\end{proof}

\section{Proof of Theorem~\ref{thm:main}}\label{sec:residues}

We first locate the irrational component by constructing, via residues, a nonzero quotient of the module $\cQ_{F_8}$ at the generic point of $C_8$, where $F_8$ is the tuple of Theorem~\ref{thm:main}. Throughout, $b=\sqrt2$ and $c=5$. For a field $K$ containing $\C$, we write
\[
\Delta_K=K[\partial_x,\partial_y]\,\delta_0\quad\text{and}\quad\Delta_0=\C[\partial_x,\partial_y]\,\delta_0,\qquad\text{where } x\delta_0=y\delta_0=0.
\]

\begin{proposition}\label{prop:sixplane}
Let $X=\C^2$ with coordinates $x,y$, and let
\[
F_8=\bigl(x,\ y,\ x+y,\ x+2y,\ x+3y,\ x+4y,\ x+by+y^2,\ x+cy+y^2\bigr),\qquad b=\sqrt2,\quad c=5 .
\]
Let
\[
S_8=\sum_{i=1}^8s_i,\quad L_8=60(s_7+1)+(49\sqrt2-70)(s_8+1),\quad \fp_8=(S_8+12,L_8)\subseteq R_8=\C[s_1,\dots,s_8].
\]
Then $\fp_8$ is a prime ideal, and, writing $K_8=\Frac(R_8/\fp_8)$, there is a surjective $\sD_X\otimes_\C K_8$-linear map
\[
\cQ_{F_8}\otimes_{R_8}K_8\onto\Delta_{K_8}.
\]
Consequently,
\[
B_{F_8}\subseteq\fp_8\qquad\text{and}\qquad C_8=V(\fp_8)\subseteq Z(B_{F_8}).
\]
The affine subspace $C_8$ has dimension six, its direction is not defined over $\Q$, and the only rational affine hyperplane containing $C_8$ is $H_{12}=V(S_8+12)$.
\end{proposition}


\begin{proof}
To simplify notation, we abbreviate $R=R_8$, 
$F=F_8$, $e=F_8^s$, and $\cQ=\cQ_{F_8}$,
and we put
\[
\lambda_i=s_i+1,\qquad \gamma=49\sqrt2-70,\qquad S=\sum_{i=1}^8s_i,\qquad L=60\lambda_7+\gamma\lambda_8 .
\]
We work over the function field
\[
K=\Frac\bigl(R/(S+12,L)\bigr),
\]
and, for comparison with the generic point of the rational hyperplane $H_{12}$, also over
\[
K_H=\Frac\bigl(R/(S+12)\bigr).
\]

Let
\[
P_5(z)=z(z+1)(z+2)(z+3)(z+4),\qquad P(z)=P_5(z)(z+b)(z+c).
\]
The seven zeros of $P$ are $0,-1,-2,-3,-4,-b,-c$. We index them by
\[
I=\{1,3,4,5,6,7,8\},\qquad (a_1,a_3,a_4,a_5,a_6,a_7,a_8)=(0,1,2,3,4,b,c),
\]
so that the puncture $-a_i$ corresponds to the $i$-th entry of $F$; these punctures are pairwise distinct. Over either $K$ or $K_H$, set $\Omega=K[z,1/P]$, with the evident replacement of $K$ by $K_H$ when needed. We also put $N=\cN\otimes_R K$.
Use the (partial) blow-up coordinates
\[
y=t,\qquad x=tz ,
\]
and write
\[
\Phi_s(z)=z^{s_1}(z+1)^{s_3}(z+2)^{s_4}(z+3)^{s_5}(z+4)^{s_6}(z+b)^{s_7}(z+c)^{s_8}.
\]
For $d\in \Z$ and $a=a(x,y,s)\in R[x,y]$, consider $m=af^de\in\cS_F$; every element of $\cN$ is of this form.
Using the (generalized) binomial expansions of $(1+\frac{t}{z+b})^{s_7+d}$ and $(1+\frac{t}{z+c})^{s_8+d}$, the pullback of $m$ to the coordinates $(t,z)$ is
\[\hat m =a(tz,t,s)\,t^{-12+8d}P^{d}\Phi_s(z)\Bigl(1+\frac{t}{z+b}\Bigr)^{s_7+d}\Bigl(1+\frac{t}{z+c}\Bigr)^{s_8+d}\in \Omega[[t]][1/t]\cdot \Phi_s(z),\]
that is,
\[\hat m=\sum_{k=-12+8d+l}^\infty\phi_k t^k,\qquad \phi_k\in \Omega\cdot\Phi_s(z),\]
where $l$ is the order of $a$ at the origin, i.e.\ the smallest total degree in $x,y$ of a monomial occurring in $a$.

Define a connection on $\Omega$ by
\[
 \nabla_s(h)=\partial_z(h\cdot \Phi_s)/ \Phi_s=\partial_z(h)+h\sum_{i\in I}\frac{s_i}{z+a_i},\qquad h\in \Omega.
\]
This is a meromorphic connection on $\mathbb P^1_K$ with eight poles, namely the seven punctures and $\infty$. Let
\[\cH_s=\Omega\,dz\big/\bigl(\nabla_s\Omega\bigr)dz \]
be the first cohomology of its de Rham complex, and for $h\in \Omega$ let $[h\,dz]$ denote the class of $h\,dz$ in $\cH_s$. 
We claim that
\begin{equation}\label{eq:basis}
[dz],\ [z\,dz],\ \dots,\ [z^5\,dz]
\end{equation}
form a $K$-basis of $\cH_s$. A pole of order $k\geq2$ of a rational one-form at $-a_i$ can be removed by subtracting a multiple of $\nabla_s\bigl((z+a_i)^{-k+1}\bigr)dz$, whose coefficient at order $k$ is $s_i-k+1\neq0$; repeating this removes every pole of order at least two. If the remaining simple-pole coefficients are $r_i$, interpolation at the seven distinct punctures gives a polynomial $h$ of degree at most six with $h(-a_i)=r_i/s_i$, and subtracting $\nabla_s h\,dz$ removes all simple poles. Hence every class has a polynomial representative. Next, if $\nabla_s h$ is a polynomial for some $h\in\Omega$, then $h$ has no finite pole, since a pole of order $k$ of $h$ at $-a_i$ would produce a pole of order $k+1$ of $\nabla_s h$ with the nonzero coefficient $s_i-k$; thus $h$ is a polynomial, and the residues $s_ih(-a_i)$ of $\nabla_s h$ must vanish, so that $h=Pq$ for a polynomial $q$. Conversely,
\[
\nabla_s(Pq)=Pq'+(P'+P\sum_{i\in I}\frac{s_i}{z+a_i})q
\]
is a polynomial. If $q$ has degree $d$, this polynomial has degree $d+6$, and its leading coefficient relative to that of $q$ equals
\[
d+7+\sum_{i\in I}s_i=d-5-s_2,
\]
which is nonzero in $K$ for every $d\geq0$. Descending subtraction of exact forms therefore reduces every polynomial one-form to one of degree at most five, while no nonzero exact polynomial one-form has degree at most five. This proves \eqref{eq:basis}. The same leading-term argument shows that $\nabla_s$ has no nonzero kernel on $\Omega$. A more conceptual explanation is that, by the Riemann--Hilbert correspondence (over $K$), $\cH_s$ is the first cohomology of a rank-one local system on $\mathbb P^1$ minus eight points. Since its $H^0$ vanishes and its Euler characteristic is $2-8=-6$, $\dim_K\cH_s=6$. 

Define  
\begin{equation}\label{eq:rho}
\rho_0(m)=[\Res_t(t\hat m)\,dz]\in\cH_s ,
\end{equation}
where $\Res_t$ extracts the coefficient of $t^{-1}\Phi_s(z)$ in the formal Laurent series $t\hat m$, i.e.\ the residue along the exceptional divisor $t=0$ after $t$-adic completion. By the chain rule, we have
\begin{equation}\label{eq:pullback}
\widehat{\partial_xm}=t^{-1}\partial_z(\hat m),\qquad \widehat{\partial_ym}=\partial_t\hat m-\frac{z}{t}\partial_z\hat m .
\end{equation}
Then
\[
\rho_0(\partial_xm)=[\nabla_s\Res_t(\hat m)\,dz]=0. 
\]
Moreover, $\Res_t(t\partial_t \hat m)=-\Res_t(\hat m)$, and hence
\begin{equation}\label{eq:rhoy}
\rho_0(\partial_ym)=\bigl[\bigl(-\Res_t(\hat m)-z\nabla_s\Res_t(\hat m)\bigr)dz\bigr]=-\bigl[\nabla_s\bigl(z\Res_t(\hat m)\bigr)dz\bigr]=0,
\end{equation}
where the second equality follows from the Leibniz rule $\nabla_s(zg)=g+z\nabla_sg$.

Let $W$ be a $K$-vector space and let $\ell_s\colon\cH_s\to W$ be $K$-linear. In $\Delta_K$, put
\[
J_{ij}=\frac{(-1)^{i+j}}{i!\,j!}\,\partial_x^i\partial_y^j\delta_0 .
\]
With negative indices interpreted as zero, we have
\begin{equation}\label{eq:Jrelations}
xJ_{ij}=J_{i-1,j},\quad yJ_{ij}=J_{i,j-1},\quad \partial_xJ_{ij}=-(i+1)J_{i+1,j},\quad \partial_yJ_{ij}=-(j+1)J_{i,j+1}.
\end{equation}
Define
\begin{equation}\label{eq:Psi}
\Psi_{\ell_s}(m)=\sum_{i,j\geq0}J_{ij}\otimes\ell_s\bigl(\rho_0(x^iy^jm)\bigr)\in\Delta_K\otimes_KW .
\end{equation}
The sum is finite because $\rho_0(x^iy^jm)=0$ whenever $i+j-12+8d+l>-2$. 
Compatibility with multiplication by $x$ and $y$ follows from \eqref{eq:Jrelations} by shifting indices. For the derivatives, we have for instance
\[
0=\rho_0\bigl(\partial_x(x^iy^jm)\bigr)=i\rho_0(x^{i-1}y^jm)+\rho_0(x^iy^j\partial_xm),
\]
and this identity, its analogue for $y$, and \eqref{eq:Jrelations} give compatibility with $\partial_x$ and $\partial_y$. Hence $\Psi_{\ell_s}$ is $\sD_X[s]$-linear on $\cN$, and it induces a $\sD_X\otimes_\C K$-linear map on $N=\cN\otimes_RK$, still denoted by $\Psi_{\ell_s}$.

For $x^py^qe$, put $k=10-p-q$. If $k\geq0$, then
\begin{equation}\label{eq:sourcemoment}
\Psi_{\ell_s}(x^py^{q}e)=\sum_{i,j\ge0,\ i+j\le k}J_{ij}\otimes\sum_{a=0}^{k-i-j}\binom{s_7}{a}\binom{s_8}{k-i-j-a}\,\ell_s\Bigl(\Bigl[\frac{z^{p+i}}{(z+b)^a(z+c)^{k-i-j-a}}\,dz\Bigr]\Bigr),
\end{equation}
and it is zero for $k<0$. In particular,
\begin{equation}\label{eq:sourcebasis}
\Psi_{\ell_s}(x^iy^{10-i}e)=J_{00}\otimes\ell_s([z^i\,dz]),\qquad 0\leq i\leq10. 
\end{equation}
The first six classes in \eqref{eq:sourcebasis} form the basis \eqref{eq:basis}. Since $\Delta_K\otimes_KW$ is generated by $J_{00}\otimes W$ as a $\sD_X\otimes_\C K$-module, it follows that if $\ell_s$ is surjective, then so is
\[\Psi_{\ell_s}\colon N\to \Delta_K\otimes_K W.\]

Similarly, for $x^py^qfe$, put $k=2-p-q$. If $k\geq0$, then
\begin{equation}\label{eq:denmoment}
\Psi_{\ell_s}(x^py^{q}fe)=\sum_{i,j\ge0,\ i+j\le k}J_{ij}\otimes\sum_{a=0}^{k-i-j}\binom{\lambda_7}{a}\binom{\lambda_8}{k-i-j-a}\,\ell_s\Bigl(\Bigl[\frac{z^{p+i}P}{(z+b)^a(z+c)^{k-i-j-a}}\,dz\Bigr]\Bigr),
\end{equation}
and it is zero for $k<0$. Moreover, the subspace $V\subseteq\cH_s$ spanned by the classes $\rho_0(x^py^{q}fe)$, for all monomials $x^py^q$, is spanned by the classes of the six functions
\begin{equation}\label{eq:sixclasses}
P,\quad zP,\quad z^2P,\quad d(z)P,\quad z\,d(z)P,\quad e_2(z)P,
\end{equation}
where
\begin{equation}\label{eq:d}
d(z)=\frac{\lambda_7}{z+b}+\frac{\lambda_8}{z+c}
\end{equation}
and
\begin{equation}\label{eq:e2}
e_2(z)=\frac12\Bigl(d(z)^2-\frac{\lambda_7}{(z+b)^2}-\frac{\lambda_8}{(z+c)^2}\Bigr).
\end{equation}
More precisely, for $p+q=2$ the classes $\rho_0(x^py^{q}fe)$ are those of $P$, $zP$ and $z^2P$; for $p+q=1$ they are those of $d(z)P$ and $z\,d(z)P$; $\rho_0(fe)=[e_2(z)P\,dz]$; and $\rho_0(x^py^{q}fe)=0$ for all other $p,q$. These six classes need not be linearly independent; we now determine the relations among them.

Put
$D(z)=(z+b)(z+c)$.
Since $P$ is invertible in $\Omega$, every element of $\Omega$ can be written as $Ph$ with $h\in\Omega$. Hence a linear relation in $\cH_s$ among the six classes \eqref{eq:sixclasses} has the form
\[
A_2(z)P+B_1(z)d(z)P+Ce_2(z)P=\nabla_s(Ph)=\Bigl(h'+h\sum_{i\in I}\frac{\lambda_i}{z+a_i}\Bigr)P ,
\]
where $h\in\Omega$, $\deg A_2\leq2$, $\deg B_1\leq1$, and $C$ is a scalar. Equivalently, dividing by $P$,
\begin{equation}\label{eq:relation}
A_2(z)+B_1(z)d(z)+Ce_2(z)=\nabla_\lambda h,\qquad\text{where }\nabla_\lambda h=h'+h\sum_{i\in I}\frac{\lambda_i}{z+a_i}.
\end{equation}

At any finite puncture, a pole of order $k$ of $h$ gives a pole of order $k+1$ on the right of \eqref{eq:relation} with coefficient $\lambda_i-k$, which is nonzero in the parameter function fields. Since the left side is regular at $0,-1,-2,-3,-4$, the function $h$ has no poles there, and its value at each of these points must vanish because the corresponding $\lambda_i$ is nonzero. At $-b$ and $-c$, the left side has poles of order at most two, so $h$ has at most simple poles there. On $S=-12$,
\[
\sum_{i\in I}\lambda_i=-4-\lambda_2 .
\]
If the polynomial part of $h$ has degree $r\geq1$, then the polynomial part of $\nabla_\lambda h$ has degree $r-1$, and its leading coefficient is $r-4-\lambda_2$ times that of $h$. Since the polynomial part of the left side of \eqref{eq:relation} has degree at most two and $r-4-\lambda_2\neq0$, we get $r\leq3$. Therefore $D(z)h(z)$ is a polynomial of degree at most five, and it is divisible by $P_5$ because $h$ vanishes at the first five punctures. Thus
\begin{equation}\label{eq:hform}
h(z)=a_0\,\frac{P_5(z)}{(z+b)(z+c)}
\end{equation}
for a scalar $a_0$.

Let $h_b$ and $h_c$ be the residues of \eqref{eq:hform} at $-b$ and $-c$, so that
\begin{equation}\label{eq:residues}
h_b=a_0\,\frac{P_5(-b)}{c-b},\qquad h_c=a_0\,\frac{P_5(-c)}{b-c}.
\end{equation}
The coefficient of $(z+b)^{-2}$ on the right of \eqref{eq:relation} is $(\lambda_7-1)h_b$, while on the left it is $C\lambda_7(\lambda_7-1)/2$. Comparing these, and doing the same at $-c$, gives
\begin{equation}\label{eq:doublepole}
h_b=\frac{C\lambda_7}{2},\qquad h_c=\frac{C\lambda_8}{2},
\end{equation}
where we have cancelled the factors $\lambda_7-1$ and $\lambda_8-1$, which are nonzero in the function fields. A nonzero pair $(a_0,C)$ can therefore exist exactly when
\begin{equation}\label{eq:compat}
P_5(-b)\lambda_8+P_5(-c)\lambda_7=0 .
\end{equation}
For $c=5$ we have $P_5(-5)=(-5)(-4)(-3)(-2)(-1)=-120$ and for $b=\sqrt{2}$
\[
P_5(-b)=140-98b=-2(49\sqrt2-70)=-2\gamma .
\]
Thus the left side of \eqref{eq:compat} equals $-2(60\lambda_7+\gamma\lambda_8)=-2L$, and the double-pole compatibility is precisely the equation $L=0$.

Over $K_H$, the element $L$ is nonzero. Then \eqref{eq:residues} and \eqref{eq:doublepole} force $a_0=C=0$, hence $h=0$, and comparing simple-pole coefficients in the remaining identity $A_2+B_1d=0$ gives $B_1(-b)\lambda_7=B_1(-c)\lambda_8=0$. Hence $B_1$ vanishes at two distinct points, so $B_1=0$, and then $A_2=0$. The six classes \eqref{eq:sixclasses} are therefore linearly independent at the generic point of $H_{12}$.

Over $K$, the condition $L=0$ gives a one-dimensional space of solutions $(a_0,C)$ of \eqref{eq:residues} and \eqref{eq:doublepole}. For such a solution the double poles cancel, so the difference $\nabla_\lambda h-Ce_2$ has at most simple poles at $-b,-c$, no other finite poles, and a polynomial part of degree at most two. A unique linear polynomial $B_1$ matches the two simple-pole coefficients through $B_1d$, and a unique quadratic polynomial $A_2$ matches the remaining polynomial part. Hence a nontrivial relation exists over $K$, and \eqref{eq:hform}, together with the argument over $K_H$, shows that the space of relations is exactly one-dimensional. Therefore $\dim_KV=5$.

Put $W=\cH_s/V$, so that $\dim_KW=1$, and use the quotient map $\ell_s\colon\cH_s\to W$ in \eqref{eq:Psi}. Every moment $\rho_0(x^iy^jfe)$ lies in $V=\ker\ell_s$, so $\Psi_{\ell_s}(fe)=0$. Since $\Psi_{\ell_s}$ is $\sD_X\otimes_\C K$-linear, it kills (the image of) $\cM\otimes_RK$. 
We thus obtain a surjection
\[
\cQ\otimes_RK\onto\Delta_K\otimes_KW.
\]
In particular, $\cQ\otimes_RK\neq0$. Since $B_{F_8}=\Ann_R(\cQ)$ and every element of $R\setminus\fp_8$ acts invertibly on $\cQ\otimes_RK$, we conclude that
\[
B_{F_8}\subseteq(S+12,L)=\fp_8 ,
\]
which proves $C_8\subseteq Z(B_{F_8})$.

Since $S+12$ and $L$ have linearly independent linear parts, $R/\fp_8$ is a polynomial ring in six variables; hence $\fp_8$ is prime and $C_8$ is an affine subspace of dimension six. Its normal space is spanned over $\C$ by
\[
v_1=(1,1,1,1,1,1,1,1)\qquad\text{and}\qquad v_2=(0,0,0,0,0,0,60,\gamma).
\]
A normal vector $\mu_1v_1+\mu_2v_2$ has rational coordinates only if $\mu_1\in \Q$ and $\mu_2=0$, since $\gamma\notin\Q$. Thus the rational vectors in the two-dimensional normal space form only the line $\Q v_1$. In particular, the normal space, and hence the direction of $C_8$, is not defined over $\Q$, and $H_{12}$ is the only rational affine hyperplane containing $C_8$.
\end{proof}

\begin{proposition}\label{prop:generation}
For the tuple $F_8$ of Proposition~\ref{prop:sixplane}, let $f_8=\prod_{i=1}^8(F_8)_i$ be the product of its entries, let $\beta=(-\tfrac32,\dots,-\tfrac32)\in\C^8$, and let
\[
\cU_\beta=\cO_{\C^2,0}[1/f_8]\cdot e_\beta
\]
be the analytic meromorphic connection germ at the origin, where $e_\beta=F_8^{\beta}$. Then
\[
\sD_{\C^2,0}\,(f_8e_\beta)=\cU_\beta .
\]
\end{proposition}

\begin{proof}
Write $f=f_8$, $\cU=\cU_\beta$, and $M=\sD_{\C^2,0}(fe_\beta)$.
Clearly,
 $\cU$, $M$ and $\cU/M$ are regular holonomic $\sD_{\C^2,0}$-modules. We will show that $\cU/M$ is supported at the origin and admits no nonzero morphism to the point module $\Delta_0$. 

Keep $P_5(z)=z(z+1)(z+2)(z+3)(z+4)$, $P(z)=P_5(z)(z+b)(z+c)$, and put $\Omega_0=\C[z,1/P]$. With $\beta_i=-\tfrac32$ for $i\in I$, define as in the proof of Proposition~\ref{prop:sixplane}
\[
\nabla_\beta h=h'+h\sum_{i\in I}\frac{\beta_i}{z+a_i},\qquad \cH_\beta=\Omega_0\,dz\big/(\nabla_\beta\Omega_0)\,dz.
\]
This gives a meromorphic connection on $\mathbb P^1$ with eight poles. 
As in the proof of Proposition~\ref{prop:sixplane}, the classes
\begin{equation}\label{eq:basisbeta}
[dz],\ [z\,dz],\ \dots,\ [z^5\,dz]
\end{equation}
form a basis of $\cH_\beta$.

We use again the coordinates $x=tz$, $y=t$. As in the proof of Proposition~\ref{prop:sixplane}, for a holomorphic germ $a=a(x,y)$ and an integer $d\geq0$, the element $m=ae_\beta/f^{d}$ becomes, in the coordinates $(t,z)$,
\begin{equation}\label{eq:scalarbeta}
\hat m=a(tz,t)\,t^{-12-8d}P^{-d}\Bigl(1+\frac{t}{z+b}\Bigr)^{-3/2-d}\Bigl(1+\frac{t}{z+c}\Bigr)^{-3/2-d}\Phi_\beta(z),
\end{equation}
where $\Phi_\beta$ denotes $\Phi_s$ with $s$ replaced by $\beta$. 
Define
\[
\rho_0(m)=[\Res_t(t\hat m)\,dz]\in\cH_\beta,
\]
where $\hat m$ is expanded as a Laurent series in $t$ using the Taylor series of $a(x,y)$, $(1+\frac{t}{z+b})^{-3/2-d}$ and $(1+\frac{t}{z+c})^{-3/2-d}$.

A homogeneous summand of total degree $n$ in the Taylor series of $a(x,y)$ contributes to $\rho_0(x^iy^jm)$ only if $i+j+n+k=10+8d$ for some $k\geq0$, so every residue depends on only finitely many Taylor coefficients. The chain rule \eqref{eq:pullback}, together with $\Res_t(t\partial_t\hat m)=-\Res_t\hat m$, gives $\rho_0(\partial_xm)=\rho_0(\partial_ym)=0$ as in \eqref{eq:rhoy}. For $\ell\in\cH_\beta^*$, define
\begin{equation}\label{eq:Psibeta}
\Psi_\ell(m)=\sum_{i,j\geq0}\ell\bigl(\rho_0(x^iy^jm)\bigr)J_{ij}\in\Delta_0 .
\end{equation}
As for \eqref{eq:Psi}, $\Psi_\ell(m)$ is a finite sum, and $\Psi_\ell\colon\cU\to\Delta_0$ is a homomorphism of analytic $\sD_{\C^2,0}$-modules.

For $m=fe_\beta$, the classes $\rho_0(x^py^qm)$ with $p+q\leq2$ are those of $P$, $zP$, $z^2P$, $d(z)P$, $z\,d(z)P$, $e_2(z)P$, where now $\lambda_7=\lambda_8=-\tfrac12$ in \eqref{eq:d} and \eqref{eq:e2}, and all other $\rho_0(x^py^qm)$ vanish. Since $\beta\notin C_8$, the argument over $K_H$ in the proof of Proposition~\ref{prop:sixplane} shows that these six classes are linearly independent, so they form a basis of $\cH_\beta$. 
Therefore
\[
\Psi_\ell(fe_\beta)=0\quad\Longrightarrow\quad\ell=0 .
\]
In other words, the map $\cH_\beta^*\to\Hom_{\sD}(\cU,\Delta_0)$, $\ell\mapsto\Psi_\ell$, is injective, so its image is six-dimensional, and restriction to $M$ is injective on this image.


Let $B_\epsilon$ be a sufficiently small ball around the origin, let
\[
j\colon B_\epsilon\setminus\{f=0\}\longrightarrow B_\epsilon,\qquad i\colon\{0\}\longrightarrow B_\epsilon ,
\]
and let $L_\beta$ be the rank-one local system of horizontal sections of $\cU$ on the complement. The regular Riemann--Hilbert equivalence \cite[Theorem\,7.2.1]{HTT08} gives $\DR(\cU)=Rj_*L_\beta[2]$ and $\DR(\Delta_0)=\C_0$. 
Since $i^{-1}$ is left adjoint to $Ri_*$ \cite[Proposition\,C.2.4]{HTT08}, we obtain
\begin{equation}\label{eq:homtop}
\Hom_{\sD}(\cU,\Delta_0)\simeq\Hom_{D^b(\C)}\bigl(i^{-1}Rj_*L_\beta[2],\C\bigr)\simeq H^2\bigl(B_\epsilon\setminus\{f=0\},L_\beta\bigr)^* .
\end{equation}
The second isomorphism uses $i^{-1}Rj_*L_\beta\simeq R\Gamma(B_\epsilon\setminus\{f=0\},L_\beta)$, together with the facts that complexes of vector spaces split into their cohomology and that vector spaces have no higher extension groups. 

We now identify the local complement. Deform the two quadratic factors $x+by+y^2$ and $x+cy+y^2$ through the families of curves
\[
x+by+\tau y^2=0,\qquad x+cy+\tau y^2=0,\qquad 0\leq\tau\leq1 .
\]
All eight tangent directions remain fixed and distinct. For $\tau>0$, the two curves meet a line $x+dy=0$ with $d\in\{0,1,2,3,4\}$ away from the origin only at $y=(d-b)/\tau$ and $y=(d-c)/\tau$, respectively, and the absolute values of these nonzero points have a positive lower bound independent of $\tau$. The two curves meet each other only at the origin, because their difference is $(b-c)y$, and each of them meets $y=0$ only at the origin. Hence a uniformly small ball contains no intersection point other than the origin throughout the deformation. Parametrize the first curve by $y=q$, $x=-bq-\tau q^2$ with $q=\rho e^{i\theta}$. Its squared norm is $\rho^2\bigl(1+|b+\tau\rho e^{i\theta}|^2\bigr)$, whose radial derivative is $2\rho(1+b^2)+O(\rho^2)>0$ for uniformly small positive $\rho$; the same computation applies with $c$ in place of $b$. Thus the two curves meet each sufficiently small sphere transversally throughout the deformation. The eight disjoint link circles vary smoothly with $\tau$, and thus, by the isotopy extension theorem, so does the complement of the link circles in the sphere. Therefore, by Milnor's local conic structure theorem, the punctured small-ball complement has the homotopy type of the complement of the eight Hopf circles determined by the tangent lines.

Since $y=0$ is among the deleted directions, this Hopf-circle complement is diffeomorphic to
\[
E\times S^1,\qquad E=\C\setminus\{0,-1,-2,-3,-4,-b,-c\};
\]
on a sphere the identification is $(x,y)=(tz,t)$ with $t=\epsilon e^{i\theta}/\sqrt{1+|z|^2}$. The local system $L_\beta$ has monodromy $-1$ around each oriented branch meridian. Its monodromy around the $S^1$-factor, which rotates all eight linear factors once, is
\[
\exp\Bigl(-2\pi i\sum_{i=1}^8\beta_i\Bigr)=1,
\]
since $\sum_i\beta_i=-12$.
It is then standard that $H^0(E,L_\beta)=0$ and hence, since $\chi(E)=-6$, that $\dim_\C H^1(E,L_\beta)=6$; by the K\"unneth formula, $\dim_\C H^2(E\times S^1,L_\beta)=\dim_\C H^1(E,L_\beta)=6$. 
By \eqref{eq:homtop},
\begin{equation}\label{eq:dimhom}
\dim_\C\Hom_{\sD}(\cU,\Delta_0)=6 .
\end{equation}

Shrink the ball so that every point of the divisor other than the origin lies on exactly one component, which is smooth there. Off the divisor, multiplication by $f$ is invertible. At a smooth point of the component $\{f_i=0\}$, the section $fe_\beta$ is a unit times $f_i^{-1/2}$, and since $-\tfrac12\notin\Z$, it generates $\cU$ there. Thus $\cU/M$ is supported at the origin.

By \eqref{eq:dimhom}, the injective map $\cH_\beta^*\to\Hom_{\sD}(\cU,\Delta_0)$ is also surjective, and thus the restriction map $\Hom_{\sD}(\cU,\Delta_0)\to \Hom_{\sD}(M,\Delta_0)$ is injective. Applying $\Hom_{\sD}(-,\Delta_0)$ to the short exact sequence $0\to M\to\cU\to\cU/M\to0$ therefore shows that $\Hom_{\sD}(\cU/M,\Delta_0)=0$. Since $\cU/M$ is holonomic and supported at the origin, Kashiwara's equivalence gives $\cU/M\simeq\Delta_0^{\oplus k}$ for some $k\geq0$; hence $\cU/M=0$, which proves $\sD_{\C^2,0}(fe_\beta)=\cU_\beta$.
\end{proof}

We can now prove Theorem~\ref{thm:main}.

\begin{proof}[Proof of Theorem~\ref{thm:main}]
By Proposition~\ref{prop:sixplane}, $C_8\subseteq Z(B_{F_8})$, and $C_8$ is irreducible of dimension six, i.e.\ of codimension two. Choose an irreducible component $Y$ of $Z(B_{F_8})$ containing $C_8$, and suppose that $C_8\subsetneq Y$. Since $B_{F_8}\neq0$, the component $Y$ then has codimension one, so by \cite[Theorem 1.5.1(ii)]{BvdVWZ21a} it is a rational affine hyperplane. By Proposition~\ref{prop:sixplane}, the only rational affine hyperplane containing $C_8$ is $H_{12}$, so $Y=H_{12}$.

Since analytification is faithfully exact (GAGA), Propositions~\ref{prop:generation} and~\ref{prop:relgen} show that $\beta\notin Z(B_{F_8,0})$; as $\beta\in H_{12}$, the hyperplane $H_{12}$ is not contained in $Z(B_{F_8,0})$. For a point $q\in \C^2$ other than the origin, by the choice of $b$ and $c$, the divisor of $f_8$ has simple normal crossings near $q$, so $Z(B_{F_8,q})$ is contained in the union of the hyperplanes $\{s_i=-1\}$; in particular, it does not contain $H_{12}$. Since $B_{F_8}$ is the intersection of the finitely many distinct local ideals $B_{F_8,q}$, $q\in \C^2$, we conclude that $H_{12}\not\subseteq Z(B_{F_8})$. This is a contradiction, and thus $C_8=Y$.
\end{proof}

\begin{remark}\label{rmk:mainrmk}
(1) The pair $(b,c)=(\sqrt{2},5)$ plays no special role. The only way $b$ and $c$
enter the proofs of Propositions~\ref{prop:sixplane} and~\ref{prop:generation} is through the two numbers $P_5(-b)$ and $P_5(-c)$. Fix
$b,c\in\mathbb{C}$,
and assume
\begin{enumerate}
\item[(i)] $b,c\notin\{0,1,2,3,4\}$ and $b\neq c$;
\item[(ii)] $P_5(-b)+P_5(-c)\neq 0$;
\item[(iii)] $P_5(-b)/P_5(-c)\notin\mathbb{Q}$.
\end{enumerate}
Then the same proofs show that
\[
C_8(b,c)=\bigl\{\,s\in\mathbb{C}^8\ \big|\ S_8+12=0,\quad
P_5(-c)(s_7+1)+P_5(-b)(s_8+1)=0\,\bigr\}
\]
is an irreducible component of codimension two of $Z(B_{F_8})$ (with $F_8$ defined using these $b,c$) whose
direction is not defined over $\mathbb{Q}$. 

(2) The constant $12$ is not arbitrary: it is forced by $r=8$ through the
following count. Since $\operatorname{Res}_t$ requires an integral radial
exponent, the construction can only be carried out on a hyperplane $\{S_8+l=0\}$ with
$l\in\mathbb{Z}$. On $\{S_8=-l\}$, the exponent of the radial parameter $t$ in the pullback of
$m=f_8F_8^{s}$ is $-l+8$, so the
scalar coefficient of the pullback of $x^py^qm$ is
$t^{\,p+q-l+8}z^p\sum_{k\ge0}E_k(z)t^k$, and $\rho_0$ extracts the coefficient
of $t^{-2}$. The surviving terms are exactly those with
\[
p+q+k=l-10 ,
\] 
so there are $\binom{l-8}{2}$ of them, while $\dim_K\cH_s=6$ because
$\cH_s$ is the first cohomology of a rank-one local system on
$\mathbb{P}^{1}$ minus $r=8$ points. Since $\{S_8+l=0\}$ must not be a component of $Z(B_{F_8})$, we have to require
\[\binom{l-8}{2}=\dim_K\cH_s=r-2,\]
which gives $l=12$. Moreover, one can construct further examples by adding linear factors to $F$ by the same method. However, one can check that making $r<8$ by deleting linear factors from $F$ is impossible since there is no relation among $\rho_0(x^py^qm)$, and thus no codimension $2$ component of $Z(F_r)$ can be produced for $r<8$. 
\end{remark}

\section{Proof of Theorem~\ref{thm:main2}}
Let $p=(-6/5,-6/5,-6/5)\in \C^3\simeq\Spec \C[s]$, and put
\[
u=s_1+6/5,\qquad v=s_2+6/5,\qquad w=s_3+6/5,
\]
and let $\mathfrak m=(u,v,w)$ be the corresponding maximal ideal.
The key claim is the following equality of ideals in the local ring $\mathbb C[u,v,w]_{\mathfrak m}$:
\begin{equation}\label{eq:keyclaim}
(B_F)_{\mathfrak m}
=\left(\bigcap_{i=1}^5P_i\right)_{\mathfrak m},
\end{equation}
where
\begin{equation}\label{eq:Bgen}
    \begin{aligned}
P_1&=(u+v+3w,w),& P_2&=(u+v+3w,v+2w),\\
P_3&=(u+2v+2w,w),& P_4&=(u+2v+2w,v+2w),\\
P_5&=(u+3v+w,v).
\end{aligned}
\end{equation}

Let $I=\bigcap_iP_i$. The lines defined by these primes have direction vectors
\[
d_1=(1,-1,0),\quad d_2=(1,2,-1),\quad d_3=(2,-1,0),\quad
d_4=(2,-2,1),\quad d_5=(1,0,-1).
\]
Any three of them are linearly independent. Thus a plane through the origin contains at most two of these lines.

The quadratic polynomial
\[
Q=u^2+3uv+2v^2+5uw+7vw+4w^2=(u+v+3w)(u+2v+2w)-w(v+2w)
\]
vanishes on all five lines. Hence $Q\in I$, and its order at $\mathfrak m$ is exactly $2$. Moreover, $Q$ is irreducible over $\mathbb C$, since the associated quadratic form has rank three. 

Suppose that $b\in B_F$ is a product of affine linear polynomials. By \eqref{eq:keyclaim}, $b\in(P_i)_\mathfrak m$ for each $i$. Since $P_i$ is prime and contained in $\mathfrak m$, contracting gives $b\in P_i$, and primality forces at least one linear factor of $b$ to lie in $P_i$. Such a factor vanishes on the entire corresponding line, and in particular at the origin of the $(u,v,w)$-coordinates.

The zero set of a linear factor vanishing at the origin is a plane through the origin, which contains at most two of the five lines. Therefore $b$ has at least three linear factors vanishing at the origin, so
$b\in\mathfrak m^3$.
If such polynomials generated $B_F$, their images would generate an ideal contained in $\mathfrak m^3\mathbb C[u,v,w]_{\mathfrak m}$. But by \eqref{eq:keyclaim} the localized ideal contains $Q$, which has order $2$, a contradiction. This proves Theorem~\ref{thm:main2}, assuming \eqref{eq:keyclaim}.


We now prove \eqref{eq:keyclaim}.
Let $g=(g_1,g_2)=(Z,X^5+Y^5+ZX^2Y^3)$ on $\C^3$ with coordinates $X,Y,Z$, and write $g^{(U,V)}=g_1^Ug_2^V$ with parameters $(U,V)$. Bahloul and Oaku \cite[Example 4 and Table 3]{BaOK} computed its global Bernstein--Sato ideal:
\[
B_g=h(U,V)\left(
(U+2)(U+3)(U+4)(U+5),\ (5V+7)(U+2),\ (5V+7)(5V+8)
\right),
\]
where
\[
h=(U+1)(V+1)^2(5V+2)(5V+3)(5V+4)(5V+6).
\]
Their computation also gives $B_{g,0}=B_g$, and \cite[Table 3]{BaOK} shows that all isolated components of this ideal lie over the origin: at every other spatial point, the zero locus of the local ideal is contained in the hyperplanes cut out by the factors of $h$.

Therefore, at $\alpha=(-2,-8/5)$ (by abuse of notation, we also write $\alpha$ for the corresponding maximal ideal of $\C[U,V]$),
\[
(B_{g,0})_\alpha=(U+2,5V+8)_\alpha,
\]
whereas $(B_{g,q})_\alpha=(1)$ for every spatial point $q\neq0$. 

For $m\in\mathbb N^2$, let
\[
B_g^m=\operatorname{Ann}_{\mathbb C[U,V]}
\left(\sD[U,V]g^{(U,V)}/\sD[U,V]g^{(U,V)+m}\right),
\]
where $\sD$ denotes the ring of algebraic differential operators on $\C^3$.
Set
$\beta=(-6,-18/5)$.
Using the third generator of $B_g$ and substitution, we know
\[\sD[U,V]_\beta\cdot g^{(U,V)}=\sD[U,V]_\beta \cdot g^{(U+2,V)}=\sD[U,V]_\beta\cdot g^{(U+2,V+2)}.\]
Then similarly
\[\sD[U,V]_\beta\cdot g^{(U,V)}=\sD[U,V]_\beta\cdot g^{(U+2,V)}=\sD[U,V]_\beta \cdot g^{(U+4,V+2)}\]
and thus
\begin{equation}\label{eq:BKO}
    (B_g^{(5,3)})_\beta=(U+6,5V+18)_\beta.
\end{equation}

For linearly independent vectors $a,b\in\mathbb N^3$ with $\sum_j a_j=5$ and $\sum_j b_j=3$, consider
\[
T=(g_1^{a_1}g_2^{b_1},g_1^{a_2}g_2^{b_2},g_1^{a_3}g_2^{b_3}).
\]
The substitution $U=a\cdot s,\ V=b\cdot s$ gives
\[
B_T=B_g^{(5,3)}\mathbb C[s].
\]

At $s=p$, the parameters are $U=-6$ and $V=-18/5$. Thus \eqref{eq:BKO} gives
\begin{equation}\label{eq:subBKO}
    (B_T)_p=(a\cdot(s-p),b\cdot(s-p))_p.
\end{equation}
All five pairs $(a_i,b_i)$ in the table are linearly independent, and for them \eqref{eq:subBKO} gives exactly the ideals $P_1,\dots,P_5$ in \eqref{eq:Bgen}. These assertions hold for analytic stalks as well.

Put $L_i=\{x=\xi_i,y=0\}$. At any point of $L_i$, every $G_j,H_j$ with $j\ne i$ is a unit, since
\[
G_j|_{L_i}=1,\qquad H_j|_{L_i}=(\xi_i-\xi_j)^5\ne0.
\]
Near $L_i$, the map
\[
(x,y,z)\longmapsto (X,Y,Z)=(x-\xi_i,y,G_i)
\]
is an analytic coordinate change: its Jacobian determinant is $A_i(x)$, equal to $1$ on $L_i$. In these coordinates, $(G_i,H_i)=g$. The preimage of the origin $X=Y=Z=0$ is
\[
q_i=(\xi_i,0,-c_i-1).
\]
Equation \eqref{eq:subBKO}, applied with $(a,b)=(a_i,b_i)$, shows that
\[
(B_{F,q_i})_\mathfrak m=(P_i)_\mathfrak m,
\]
and that at every other point of $L_i$ the localized local ideal is the unit ideal.

We next prove that the reduced divisor $\prod_i G_iH_i=0$ has simple normal crossings outside $\bigcup_i L_i$. This rules out additional singularities that could affect \eqref{eq:keyclaim}.

\subsection*{\texorpdfstring{On the planes $x=\xi_j$}{On the planes x = xi\_j}}

On $x=\xi_j$, for $i\ne j$, $G_i=1$, and, with $d=\xi_j-\xi_i\ne0$,
\[
H_i=d^5+y^5+d^2y^3=d^5P(y/d),\qquad P(t)=1+t^3+t^5.
\]
The polynomial $P$ has five distinct nonzero roots. Indeed, a common root of $P$ and $P'=t^2(3+5t^2)$ would satisfy $t^2=-3/5$ and $1-6t/25=0$, which are incompatible.

The root sets $(\xi_j-\xi_i)\cdot\operatorname{Roots}(P)$, for different $i\neq j$, are disjoint: an equality between two such roots would be a nontrivial linear relation with algebraic coefficients among the algebraically independent $\xi$'s. Thus at most one of these $H_i$ vanishes at a given point of the plane, and it has nonzero $y$-derivative there, while its $z$-derivative vanishes on the plane.

Moreover, $H_j=y^5$ on this plane, so it does not vanish outside $L_j$. The only other factor that may vanish is $G_j$, whose $z$-derivative is $1$ on the plane. If $G_j$ and some $H_i$ vanish at the same point, their gradients there are linearly independent. This proves the normal-crossing assertion on these planes away from $L_j$.

\subsection*{Off those planes}

Let $\Omega=\{x\ne\xi_i\text{ for all }i\}$, and write
\[
h_i=(x-\xi_i)^5+y^5,\qquad k_i=(x-\xi_i)^2y^3.
\]
On $\Omega$, each $A_i$ is nonzero, and a zero of $H_i$ has $y\neq0$, hence $k_i\neq0$. Consequently, the equations of the two divisors with index $i$ can be written as
\begin{equation}\label{eq:def11}
    \begin{array}{ll}
G_i=0:&c_i=-z-1/A_i,\\[2pt]
H_i=0:&c_i=-z-1/A_i-h_i/(A_ik_i).
\end{array}
\end{equation}
Selecting both divisors at index $i$ is equivalent to selecting $h_i=0$ together with the first equation in \eqref{eq:def11}.

Each locus $h_i=0$ is the union of the five planes $y=\zeta(x-\xi_i)$, $\zeta^5=-1$, whose common line $L_i$ does not meet $\Omega$. For different indices, equal slopes give disjoint parallel planes; unequal slopes meet transversely. No three planes from three different indices meet: concurrence would give a nontrivial algebraic-coefficient linear relation among the corresponding $\xi_i$'s. Thus for a set $J$ of doubled indices, the locus
\[
\mathcal H_J=\Omega\cap\bigcap_{i\in J}\{h_i=0\}
\]
is smooth of codimension $|J|$ if $|J|\le2$, and is empty if $|J|\ge3$.

Fix a selection of divisors, and let $I$ be its set of indices and $J\subseteq I$ its set of doubled indices. On the appropriate open subset of $\mathcal H_J$, the equations \eqref{eq:def11} define the fiber over $(c_i)_{i\in I}$ of a regular map to $\mathbb A^{|I|}$, defined over $\overline{\mathbb Q}(\xi_1,\ldots,\xi_5)$. By generic smoothness in characteristic zero, its fiber over this algebraically independent parameter tuple is smooth of dimension
\[
3-|J|-|I|,
\]
or is empty if the map is not dominant. Equivalently, all the selected divisor equations have independent differentials at every common zero. There are finitely many selections, and the chosen algebraically independent $c_i$'s are generic for all of them. This proves the normal-crossing assertion on $\Omega$.

At a simple normal crossing point, the tuple entries are analytic units times monomials in local coordinates. If a reduced component has exponent vector $e$, its contribution to the simultaneous-shift Bernstein--Sato generator is
\[
\prod_{k=1}^{\sum e_j}(e\cdot s+k).
\]
This is the elementary formula for monomials (see also \cite[Example 4.4]{Bud15}). In the present construction, $e$ is one of the $a_i$ or $b_i$, whose entries sum to $5$ or $3$, respectively. At $s=p$, these factors are $-6+k$, $1\le k\le5$, or $-18/5+k$, $1\le k\le3$. All these factors are nonzero, so the local ideal becomes the unit ideal after localizing at $p$.

Therefore, the only localized local ideals that are not the unit ideal are the five ideals $(P_i)_{\mathfrak m}$. The global Bernstein--Sato ideal is the intersection of the local ideals, and there are only finitely many distinct local ideals. Localization commutes with this finite intersection. This proves \eqref{eq:keyclaim}.

\bibliographystyle{amsalpha}
\bibliography{cexbib}

\end{document}